\documentclass[11pt,a4paper]{article}

\usepackage[utf8]{inputenc}
\usepackage[T1]{fontenc}
\usepackage{lmodern}
\usepackage{amsmath,amssymb,amsfonts,amsthm}
\usepackage{geometry}
\usepackage{cite}
\usepackage[hidelinks]{hyperref}

\newtheorem{theorem}{Theorem}[section]
\newtheorem{corollary}[theorem]{Corollary}

\theoremstyle{remark}
\newtheorem*{remark}{Remark}

\theoremstyle{definition}
\newtheorem{example}[theorem]{Example}

\newcommand{\Rad}{\operatorname{Rad}}
\newcommand{\ltr}{\operatorname{ltr}}

\title{\textbf{Parallel Radical and Screen Distributions on Lightlike
Hypersurfaces of Indefinite Statistical Manifolds}}

\author{
Bur\c{c}in Do\u{g}an\\
\small Department of Engineering Basic Sciences,\\
\small Faculty of Engineering and Natural Sciences,\\
\small Malatya Turgut \"Ozal University, Malatya, T\"urkiye\\
\small E-mail: burcin.dogan@ozal.edu.tr\\
\small ORCID: 0000-0001-8386-213X
}

\date{}

\begin{document}

\maketitle

\begin{abstract}
We study the radical and screen distributions of a lightlike hypersurface of
an indefinite statistical manifold with respect to the induced connections
arising from the ambient dual affine connections. Necessary and sufficient
conditions are obtained for each distribution to be parallel with respect to
both induced connections. These conditions are expressed through the screen
shape operators, the local screen fundamental forms, the mean induced
connection, and the tangential difference tensor. When both distributions
are parallel, we characterize the screen projection and obtain a block
decomposition of the tangential difference tensor, whose mixed
radical--screen part vanishes. We also obtain equivalent conditions for the
induced connections to be metric connections. In the metric case, the induced
screen connections coincide with the Levi--Civita connection on each integral
manifold of the screen distribution, and the curvature tensors of the three
induced connections agree on screen directions. We further determine the
action of these curvature tensors on the radical direction. An explicit
example shows that both distributions may be parallel although neither
induced connection is a metric connection and their curvature tensors need
not vanish.
\end{abstract}


\medskip


\section{Introduction}\label{sec:introduction}

Statistical geometry considers a metric together with a torsion-free affine
connection satisfying the Codazzi condition. Metric duality determines a
second torsion-free connection, and their mean is the Levi--Civita connection.
The foundations of this theory were established by Amari \cite{Amari1985} and
Lauritzen \cite{Lauritzen1987}. Vos \cite{Vos1989} developed the fundamental
equations for statistical submanifolds, and Furuhata \cite{Furuhata2009}
studied hypersurfaces of statistical manifolds.

If the ambient metric is indefinite, the induced metric on a hypersurface may
be degenerate. Then the tangent bundle meets its orthogonal complement along
the radical distribution, and the usual non-degenerate tangent--normal
splitting is no longer available. For the semi-Riemannian background see
O'Neill \cite{ONeill1983}. The systematic theory of lightlike submanifolds
uses a non-degenerate screen distribution and the corresponding lightlike
transversal bundle; see Duggal and Bejancu \cite{DuggalBejancu1996}, Duggal
and Jin \cite{DuggalJin2007}, and Duggal and \c{S}ahin
\cite{DuggalSahin2010}.

Lightlike submanifolds of indefinite statistical manifolds were considered by
Jain, Singh, and Kumar \cite{JainSinghKumar2020} and by Kazemi Balgeshir and
Salahvarzi \cite{KazemiSalahvarzi2020}. Rani and Kaur
\cite{RaniKaur2021} studied lightlike hypersurfaces in this setting.
Further results include lightlike geometry in indefinite Sasakian statistical
manifolds \cite{Bahadir2021}, curvature identities for lightlike
hypersurfaces \cite{BahadirEtAl2022}, and relations among the induced objects
associated with the dual connections \cite{BahadirTripathi2023}. In
particular, Bahad{\i}r and Tripathi showed that the induced screen
connections are dual with respect to the non-degenerate screen metric and
examined parallel and integrable screen distributions. More recent studies treat additional almost contact and Kenmotsu
statistical structures \cite{BhattiKaur2024,AhmadAlamAli2025}.

The radical and screen distributions are the two complementary components of
the tangent splitting of a lightlike hypersurface. In the statistical setting,
the dual ambient connections induce, in general, distinct torsion-free
connections on the hypersurface. Consequently, the condition that a
distribution be parallel for one induced connection does not a priori imply
the corresponding condition for the other. It is therefore natural to ask
when the radical and screen distributions are preserved by both induced
connections.

Earlier work already treats parallel radical and screen distributions through
the individual induced connections
\cite{JainSinghKumar2020,RaniKaur2021,BahadirTripathi2023}. We instead study
the geometry that arises when both induced connections preserve these two
distributions simultaneously. This simultaneous preservation is expressed in
terms of the mean induced connection and the tangential difference tensor. In
particular, simultaneous preservation of the radical and screen distributions
is shown to be equivalent to the screen projection being parallel with respect
to both induced connections, and equivalently to its being parallel with
respect to the mean induced connection together with a commutation relation for
the difference tensor. This yields a block decomposition of the difference
tensor and shows that its mixed radical--screen part vanishes. Consequences for
metric connections and curvature are obtained, including coincidence of the
three induced curvature tensors on screen directions in the metric case. An
explicit example shows that both distributions may be parallel although the
two induced connections are not metric connections.

Section~\ref{sec:preliminaries} recalls the notation and formulas used in the
paper. The main results are given in Section~\ref{sec:parallel}, followed by
a short conclusion.

\section{Preliminaries}\label{sec:preliminaries}

Let $(\overline{M},\overline{g})$ be a semi-Riemannian manifold and
$\overline{\nabla}$ a torsion-free affine connection on $\overline{M}$.
The pair $(\overline{\nabla},\overline{g})$ is called a statistical
structure if
\begin{equation}\label{eq:codazzi}
(\overline{\nabla}_{X}\overline{g})(Y,Z)
=
(\overline{\nabla}_{Y}\overline{g})(X,Z)
\end{equation}
for all vector fields $X,Y,Z$ on $\overline{M}$. The connection
$\overline{\nabla}^{*}$ dual to $\overline{\nabla}$ with respect to
$\overline{g}$ is determined by
\begin{equation}\label{eq:duality}
X\overline{g}(Y,Z)
=
\overline{g}(\overline{\nabla}_{X}Y,Z)
+
\overline{g}(Y,\overline{\nabla}^{*}_{X}Z).
\end{equation}
If $(\overline{\nabla},\overline{g})$ is statistical, then
$\overline{\nabla}^{*}$ is also torsion-free and
$(\overline{\nabla}^{*},\overline{g})$ is a statistical structure.
Moreover, the Levi--Civita connection $\overline{\nabla}^{0}$ of
$\overline{g}$ satisfies
\begin{equation}\label{eq:levicivita-mean}
\overline{\nabla}^{0}
=
\frac{1}{2}
\left(
\overline{\nabla}
+
\overline{\nabla}^{*}
\right).
\end{equation}
We refer to
$(\overline{M},\overline{g},\overline{\nabla},
\overline{\nabla}^{*})$
as an indefinite statistical manifold
\cite{Amari1985,Lauritzen1987,Vos1989,Furuhata2009}.

Let $M$ be a lightlike hypersurface of
$(\overline{M},\overline{g})$, and denote by $g$ the degenerate metric
induced on $M$. Since $M$ has codimension one, its radical distribution
is the rank-one bundle
\begin{equation}\label{eq:radical}
\Rad(TM)=TM\cap TM^{\perp}=TM^{\perp}.
\end{equation}
A screen distribution $S(TM)$ is a non-degenerate complementary
distribution to $\Rad(TM)$ in $TM$. Thus,
\begin{equation}\label{eq:tangent-decomposition}
TM=S(TM)\mathbin{\perp}\Rad(TM).
\end{equation}
Choose a nowhere-zero local section $\xi$ of $\Rad(TM)$. Corresponding
to the choice of $S(TM)$, there exists locally a unique lightlike
transversal vector field $N$ satisfying
\begin{equation}\label{eq:null-frame}
\overline{g}(N,N)=0,\qquad
\overline{g}(N,\xi)=1,\qquad
\overline{g}(N,W)=0
\end{equation}
for every $W\in\Gamma(S(TM))$. The line bundle generated by $N$ is
denoted by $\ltr(TM)$. Consequently,
\begin{equation}\label{eq:ambient-decomposition}
T\overline{M}|_{M}
=
S(TM)\mathbin{\perp}
\bigl(\Rad(TM)\oplus\ltr(TM)\bigr)
=
TM\oplus\ltr(TM).
\end{equation}
These constructions are standard in lightlike hypersurface geometry
\cite{DuggalBejancu1996,DuggalJin2007,DuggalSahin2010}.

Let $P$ denote the projection of $TM$ onto $S(TM)$ with respect to
\eqref{eq:tangent-decomposition}, and put
\begin{equation}\label{eq:eta}
\eta(X)=\overline{g}(X,N),\qquad X\in\Gamma(TM).
\end{equation}
Then every tangent vector field admits the decomposition
\begin{equation}\label{eq:projection}
X=PX+\eta(X)\xi.
\end{equation}

For $X,Y\in\Gamma(TM)$, the Gauss formulas associated with the dual
ambient connections are written as
\begin{align}
\overline{\nabla}_{X}Y
&=
\nabla_{X}Y+B(X,Y)N,
\label{eq:gauss}\\
\overline{\nabla}^{*}_{X}Y
&=
\nabla^{*}_{X}Y+B^{*}(X,Y)N,
\label{eq:gauss-star}
\end{align}
where $\nabla$ and $\nabla^{*}$ are the induced torsion-free
connections on $M$, and $B$ and $B^{*}$ are the corresponding second
fundamental forms. Since the ambient connections are torsion-free,
$B$ and $B^{*}$ are symmetric. Following the convention used for lightlike hypersurfaces of statistical
manifolds in \cite{BahadirTripathi2023}, the Weingarten formulas are
\begin{align}
\overline{\nabla}_{X}N
&=
-A^{*}_{N}X+\tau^{*}(X)N,
\label{eq:weingarten}\\
\overline{\nabla}^{*}_{X}N
&=
-A_{N}X+\tau(X)N,
\label{eq:weingarten-star}
\end{align}
where $A^{*}_{N}$ and $A_{N}$ are the corresponding Weingarten mappings
and $\tau^{*},\tau$ are one-forms on $M$.

The induced connections can be decomposed further along the screen
distribution. For $X\in\Gamma(TM)$ and $W\in\Gamma(S(TM))$, write
\begin{align}
\nabla_{X}W
&=
\nabla^{S}_{X}W+C(X,W)\xi,
\label{eq:screen-gauss}\\
\nabla^{*}_{X}W
&=
\nabla^{*S}_{X}W+C^{*}(X,W)\xi,
\label{eq:screen-gauss-star}
\end{align}
where $\nabla^{S}$ and $\nabla^{*S}$ are the induced connections on $S(TM)$,
while $C,C^{*}$ are the corresponding local screen fundamental forms.
For the radical vector field $\xi$, the tangential parts of the
derivatives are expressed as
\begin{align}
\nabla_{X}\xi
&=
-A_{\xi}X-\tau(X)\xi,
\label{eq:xi-derivative}\\
\nabla^{*}_{X}\xi
&=
-A^{*}_{\xi}X-\tau^{*}(X)\xi,
\label{eq:xi-derivative-star}
\end{align}
where $A_{\xi}$ and $A^{*}_{\xi}$ take values in $S(TM)$ and are the
corresponding screen shape operators.

The duality relation \eqref{eq:duality} yields, for
$X,Y,Z\in\Gamma(TM)$,
\begin{align}
Xg(Y,Z)
={}&
g(\nabla_{X}Y,Z)
+
g(Y,\nabla^{*}_{X}Z) \nonumber\\
&\quad
+B(X,Y)\eta(Z)
+
B^{*}(X,Z)\eta(Y).
\label{eq:induced-duality}
\end{align}
Thus, in contrast with the non-degenerate statistical submanifold case,
the induced connections $\nabla$ and $\nabla^{*}$ need not form a dual
pair with respect to the degenerate metric $g$. On the screen
distribution, however, \eqref{eq:induced-duality} reduces to
\begin{equation}\label{eq:screen-duality}
Xg(U,V)
=
g(\nabla^{S}_{X}U,V)
+
g(U,\nabla^{*S}_{X}V),
\qquad
U,V\in\Gamma(S(TM)).
\end{equation}
Hence $\nabla^{S}$ and $\nabla^{*S}$ are dual with respect to the
non-degenerate metric induced on $S(TM)$.

The second fundamental forms and shape operators satisfy the relations
\begin{align}
B(X,W)
&=
g(A^{*}_{\xi}X,W),
&
B^{*}(X,W)
&=
g(A_{\xi}X,W),
\label{eq:B-shape}\\
C(X,W)
&=
g(A_{N}X,W),
&
C^{*}(X,W)
&=
g(A^{*}_{N}X,W),
\label{eq:C-shape}
\end{align}
for $X\in\Gamma(TM)$ and $W\in\Gamma(S(TM))$. In addition,
\begin{equation}\label{eq:B-xi}
B(X,\xi)+B^{*}(X,\xi)=0.
\end{equation}
Related formulas for the induced geometric objects can be found in
\cite{JainSinghKumar2020,KazemiSalahvarzi2020,RaniKaur2021,
BahadirTripathi2023}.

For later use, define the mean induced connection and the
difference tensor by
\begin{equation}\label{eq:mean-difference}
\nabla^{0}=\frac{1}{2}\left(\nabla+\nabla^{*}\right),
\qquad
K(X,Y)=\frac{1}{2}\left(\nabla_{X}Y-\nabla^{*}_{X}Y\right).
\end{equation}
Then
\begin{equation}\label{eq:connection-decomposition}
\nabla_{X}Y=\nabla^{0}_{X}Y+K(X,Y),
\qquad
\nabla^{*}_{X}Y=\nabla^{0}_{X}Y-K(X,Y).
\end{equation}
Since $\nabla$ and $\nabla^{*}$ are torsion-free, $K$ is symmetric in its
two entries. The connection $\nabla^{0}$ is the induced connection obtained from the
Levi--Civita connection $\overline{\nabla}^{0}$ along $M$. If
$B^{0}=\frac12(B+B^{*})$, then
\begin{equation}\label{eq:nabla0-g}
(\nabla^{0}_{X}g)(Y,Z)
=
B^{0}(X,Y)\eta(Z)+B^{0}(X,Z)\eta(Y),
\qquad B^{0}(X,\xi)=0.
\end{equation}
Hence $\nabla^{0}$ is a metric connection if and only if $B^{0}=0$.
Since the induced metric is degenerate, $\nabla^{0}$ is not an intrinsic
Levi--Civita connection of $(M,g)$.

\section{Parallel Radical and Screen Distributions}
\label{sec:parallel}

Let $M$ be a lightlike hypersurface of an indefinite statistical manifold
$(\overline{M},\overline{g},\overline{\nabla},
\overline{\nabla}^{*})$. A distribution $\mathcal D$ on $M$ is said to
be parallel with respect to $\nabla$ if
$\nabla_XY\in\Gamma(\mathcal D)$ for all
$X\in\Gamma(TM)$ and $Y\in\Gamma(\mathcal D)$; the definition for
$\nabla^{*}$ is analogous. Related conditions for radical and screen distributions with respect to
individual induced connections appear in
\cite{JainSinghKumar2020,RaniKaur2021,BahadirTripathi2023}. Here we keep
the two induced connections distinct and focus on the conditions that hold for
both of them, comparing these common conditions through $\nabla^{0}$ and $K$.

\begin{theorem}\label{thm:radical-both}
Let $M$ be a lightlike hypersurface of an indefinite statistical manifold
$(\overline{M},\overline{g},\overline{\nabla},\overline{\nabla}^{*})$.
Denote by $\nabla$ and $\nabla^{*}$ the connections on $M$ induced by
$\overline{\nabla}$ and $\overline{\nabla}^{*}$, respectively. The following
assertions are equivalent:
\begin{enumerate}
\item $\Rad(TM)$ is parallel with respect to both $\nabla$ and
$\nabla^{*}$;
\item $A_{\xi}=A_{\xi}^{*}=0$;
\item $\nabla^{0}$ is a metric connection and
$K(X,\xi)\in\Gamma(\Rad(TM))$ for all $X\in\Gamma(TM)$;
\item $B(X,W)=B^{*}(X,W)=0$ for all $X\in\Gamma(TM)$ and
$W\in\Gamma(S(TM))$.
\end{enumerate}
\end{theorem}

\begin{proof}
Equations~\eqref{eq:xi-derivative} and
\eqref{eq:xi-derivative-star} give (1)$\Leftrightarrow$(2). For $W\in\Gamma(S(TM))$, \eqref{eq:B-shape} shows that
$2B^{0}(X,W)$ equals $g((A_{\xi}+A_{\xi}^{*})X,W)$. Together with
$B^{0}(X,\xi)=0$ in \eqref{eq:nabla0-g} and the fact that the metric on
$S(TM)$ is non-degenerate, this shows that $B^{0}=0$ if and only if
$A_{\xi}+A_{\xi}^{*}=0$. Thus, by \eqref{eq:nabla0-g},
$\nabla^{0}$ is a metric connection if and only if
$A_{\xi}+A_{\xi}^{*}=0$. Moreover,
\begin{equation}\label{eq:K-xi}
K(X,\xi)
=-\frac12(A_{\xi}-A_{\xi}^{*})X
-\frac12\{\tau(X)-\tau^{*}(X)\}\xi.
\end{equation}
Hence (2)$\Leftrightarrow$(3). Finally, \eqref{eq:B-shape} gives
$A_{\xi}=0$ if and only if $B^{*}(X,W)=0$, and
$A_{\xi}^{*}=0$ if and only if $B(X,W)=0$. Thus
(2)$\Leftrightarrow$(4).
\end{proof}

\begin{corollary}\label{cor:radical-totally-geodesic}
Let $M$ be a lightlike hypersurface of
$(\overline{M},\overline{g},\overline{\nabla},\overline{\nabla}^{*})$ with
induced connections $\nabla$ and $\nabla^{*}$. If $\Rad(TM)$ is parallel
with respect to both induced connections, then
\begin{equation}\label{eq:B-radical-form}
B(X,Y)=\eta(X)\eta(Y)B(\xi,\xi),
\qquad
B^{*}(X,Y)=-\eta(X)\eta(Y)B(\xi,\xi).
\end{equation}
Consequently, $M$ is totally geodesic with respect to
$\overline{\nabla}$ if and only if it is totally geodesic with respect
to $\overline{\nabla}^{*}$; either condition is equivalent to
$B(\xi,\xi)=0$.
\end{corollary}

\begin{proof}
Write $X=PX+\eta(X)\xi$ and $Y=PY+\eta(Y)\xi$. By
Theorem~\ref{thm:radical-both}, $B$ and $B^{*}$ vanish whenever one
entry lies in $S(TM)$. Their symmetry then leaves only the
radical--radical terms, and \eqref{eq:B-xi} gives
\eqref{eq:B-radical-form}. The last assertion
follows from the definitions of a totally geodesic lightlike hypersurface
with respect to the two ambient connections.
\end{proof}

\begin{remark}
The identities in \eqref{eq:B-shape} are used with the second entry in
$S(TM)$. Thus the vanishing of $A_{\xi}$ and $A_{\xi}^{*}$ eliminates all
components of $B$ and $B^{*}$ with one entry in $S(TM)$, but does not
by itself eliminate the remaining radical--radical component. This is why the additional condition
$B(\xi,\xi)=0$ appears in Corollary~\ref{cor:radical-totally-geodesic}.
\end{remark}

The separate conditions involving $C$ and $C^{*}$ are closely related to
those considered in \cite{RaniKaur2021,BahadirTripathi2023}. The next result keeps the
two induced connections distinct and records their common condition.

\begin{theorem}\label{thm:screen-both}
Consider a lightlike hypersurface $M$ of the indefinite statistical manifold
$(\overline{M},\overline{g},\overline{\nabla},\overline{\nabla}^{*})$, and
write $\nabla$ and $\nabla^{*}$ for the induced connections. The following
assertions are equivalent:
\begin{enumerate}
\item $S(TM)$ is parallel with respect to both $\nabla$ and
$\nabla^{*}$;
\item $C=C^{*}=0$ on $TM\times S(TM)$;
\item $A_NX,A_N^{*}X\in\Gamma(\Rad(TM))$ for all
$X\in\Gamma(TM)$;
\item $S(TM)$ is parallel with respect to $\nabla^{0}$ and
$K(X,W)\in\Gamma(S(TM))$ for all $X\in\Gamma(TM)$ and
$W\in\Gamma(S(TM))$.
\end{enumerate}
\end{theorem}

\begin{proof}
Equations~\eqref{eq:screen-gauss}--\eqref{eq:screen-gauss-star} give
(1)$\Leftrightarrow$(2). By \eqref{eq:C-shape}, and since the metric induced on $S(TM)$ is
non-degenerate, (2)$\Leftrightarrow$(3). Finally,
\begin{align}
\nabla_X^{0}W
&=\frac12(\nabla_X^{S}W+\nabla_X^{*S}W)
 +\frac12\{C(X,W)+C^{*}(X,W)\}\xi,\label{eq:nabla0-screen}\\
K(X,W)
&=\frac12(\nabla_X^{S}W-\nabla_X^{*S}W)
 +\frac12\{C(X,W)-C^{*}(X,W)\}\xi.\label{eq:K-screen}
\end{align}
Thus (4) holds if and only if $C+C^{*}=0$ and $C-C^{*}=0$, which is
precisely (2).
\end{proof}

We now consider the two distributions simultaneously. The screen
projection gives a concise formulation of the resulting condition.

\begin{theorem}\label{thm:both-distributions}
Let $M$, with induced connections $\nabla$ and $\nabla^{*}$, be a lightlike
hypersurface of the indefinite statistical manifold
$(\overline{M},\overline{g},\overline{\nabla},\overline{\nabla}^{*})$.
The following assertions are equivalent:
\begin{enumerate}
\item $\Rad(TM)$ and $S(TM)$ are parallel with respect to both
$\nabla$ and $\nabla^{*}$;
\item the screen projection $P$ is parallel with respect to both
$\nabla$ and $\nabla^{*}$;
\item $A_{\xi}=A_{\xi}^{*}=0$ and $C=C^{*}=0$ on
$TM\times S(TM)$;
\item $P$ is parallel with respect to $\nabla^{0}$ and
\begin{equation}\label{eq:K-commutes-P}
K(X,PY)=P K(X,Y)
\end{equation}
for all $X,Y\in\Gamma(TM)$.
\end{enumerate}
\end{theorem}

\begin{proof}
Using \eqref{eq:projection}, \eqref{eq:screen-gauss}, and
\eqref{eq:xi-derivative}, we obtain
\begin{equation}\label{eq:nabla-P}
(\nabla_XP)Y=C(X,PY)\xi+\eta(Y)A_{\xi}X.
\end{equation}
Similarly,
\begin{equation}\label{eq:nabla-star-P}
(\nabla_X^{*}P)Y=C^{*}(X,PY)\xi+\eta(Y)A_{\xi}^{*}X.
\end{equation}
By Theorems~\ref{thm:radical-both} and \ref{thm:screen-both},
(1) is equivalent to (3). On the other hand, taking first
$Y\in\Gamma(S(TM))$ and then $Y=\xi$ in
\eqref{eq:nabla-P}--\eqref{eq:nabla-star-P} shows that (2) is
equivalent to (3). From \eqref{eq:connection-decomposition},
\begin{align*}
(\nabla_XP)Y
&=(\nabla_X^{0}P)Y+K(X,PY)-P K(X,Y),\\
(\nabla_X^{*}P)Y
&=(\nabla_X^{0}P)Y-K(X,PY)+P K(X,Y).
\end{align*}
They vanish simultaneously if and only if $\nabla^{0}P=0$ and
\eqref{eq:K-commutes-P} holds. Hence (2)$\Leftrightarrow$(4).
\end{proof}

\begin{theorem}\label{thm:K-splitting}
Let $M$ be a lightlike hypersurface of
$(\overline{M},\overline{g},\overline{\nabla},\overline{\nabla}^{*})$.
Suppose that the connections $\nabla$ and $\nabla^{*}$ induced by
$\overline{\nabla}$ and $\overline{\nabla}^{*}$ preserve both
$\Rad(TM)$ and $S(TM)$. Then the mixed radical--screen part of the
difference tensor vanishes:
\begin{equation}\label{eq:K-mixed-zero}
K(W,\xi)=K(\xi,W)=0,
\qquad W\in\Gamma(S(TM)).
\end{equation}
Moreover, for all $X,Y\in\Gamma(TM)$,
\begin{equation}\label{eq:K-block-decomposition}
K(X,Y)=K(PX,PY)+\eta(X)\eta(Y)K(\xi,\xi),
\end{equation}
where $K(PX,PY)\in\Gamma(S(TM))$ and
$K(\xi,\xi)\in\Gamma(\Rad(TM))$. In addition,
\begin{equation}\label{eq:tau-difference-screen}
(\tau-\tau^{*})|_{S(TM)}=0,
\qquad
\tau-\tau^{*}=\{\tau(\xi)-\tau^{*}(\xi)\}\eta.
\end{equation}
\end{theorem}

\begin{proof}
By Theorem~\ref{thm:both-distributions},
$K(X,PY)=PK(X,Y)$. For $W\in\Gamma(S(TM))$, taking $X=\xi$ and
$Y=W$ gives $K(\xi,W)=PK(\xi,W)$, so $K(\xi,W)$ is screen-valued.
Taking instead $X=W$ and $Y=\xi$ gives $PK(W,\xi)=0$. Since $K$ is
symmetric, $K(\xi,W)=K(W,\xi)$ is therefore both screen-valued and
radical-valued, and hence it vanishes. This proves
\eqref{eq:K-mixed-zero}.

Expanding $X=PX+\eta(X)\xi$ and $Y=PY+\eta(Y)\xi$ and using
\eqref{eq:K-mixed-zero} gives \eqref{eq:K-block-decomposition}. The first
term is screen-valued by \eqref{eq:K-commutes-P}, whereas setting
$Y=\xi$ in that relation gives $PK(X,\xi)=0$, and in particular
$K(\xi,\xi)\in\Gamma(\Rad(TM))$. Finally, Theorem~\ref{thm:radical-both}
and \eqref{eq:K-xi} imply
$K(W,\xi)=-\frac12\{\tau(W)-\tau^{*}(W)\}\xi$ for
$W\in\Gamma(S(TM))$. Together with \eqref{eq:K-mixed-zero}, this proves
the first identity in \eqref{eq:tau-difference-screen}; the second follows
from $X=PX+\eta(X)\xi$ and $\eta(\xi)=1$.
\end{proof}

\begin{corollary}\label{cor:connection-decomposition}
Let $M$ be a lightlike hypersurface of an indefinite statistical manifold
$(\overline{M},\overline{g},\overline{\nabla},\overline{\nabla}^{*})$.
Write $\nabla$ and $\nabla^{*}$ for its induced connections. Assume that
both $\Rad(TM)$ and $S(TM)$ are parallel with respect to these connections.
Then
\begin{align}
\nabla_XY
&=\nabla_X^{S}(PY)
 +\{X(\eta(Y))-\tau(X)\eta(Y)\}\xi,
\label{eq:nabla-under-parallel}\\
\nabla_X^{*}Y
&=\nabla_X^{*S}(PY)
 +\{X(\eta(Y))-\tau^{*}(X)\eta(Y)\}\xi.
\label{eq:nabla-star-under-parallel}
\end{align}
Moreover,
$(\nabla_X\eta)(Y)=\tau(X)\eta(Y)$ and
$(\nabla_X^{*}\eta)(Y)=\tau^{*}(X)\eta(Y)$.
\end{corollary}

\begin{proof}
The first two formulas follow from \eqref{eq:projection} and the
vanishing conditions in Theorem~\ref{thm:both-distributions}. Applying
$\eta$ to these formulas and using the definition of the covariant derivative
of a $1$-form gives the last two identities.
\end{proof}

\begin{theorem}\label{thm:metric-connections}
Let $M$ be a lightlike hypersurface of an indefinite statistical manifold
$(\overline{M},\overline{g},\overline{\nabla},\overline{\nabla}^{*})$
whose induced connections are denoted by $\nabla$ and $\nabla^{*}$.
Assume that $\Rad(TM)$ and $S(TM)$ are parallel with respect to both
induced connections. Then the following assertions are equivalent:
\begin{enumerate}
\item $\nabla$ is a metric connection;
\item $\nabla^{*}$ is a metric connection;
\item $K(X,W)=0$ for all $X\in\Gamma(TM)$ and
$W\in\Gamma(S(TM))$;
\item $\nabla_X^{S}W=\nabla_X^{*S}W$ for all
$X\in\Gamma(TM)$ and $W\in\Gamma(S(TM))$.
\end{enumerate}
\end{theorem}

\begin{proof}
Theorem~\ref{thm:radical-both} implies that $\nabla^{0}$ is a metric
connection, and Theorem~\ref{thm:both-distributions} implies that
$K(X,W)$ is screen-valued. For $U,V\in\Gamma(S(TM))$, using
\eqref{eq:connection-decomposition}, the fact that $\nabla^{0}$ is a metric
connection, and \eqref{eq:screen-duality}, we obtain
$g(K(X,U),V)=g(U,K(X,V))$. Therefore
\begin{equation}\label{eq:metric-K-screen}
(\nabla_Xg)(U,V)=-2g(K(X,U),V),
\qquad
(\nabla_X^{*}g)(U,V)=2g(K(X,U),V).
\end{equation}
The radical distribution is parallel with respect to both induced
connections, so the covariant derivatives of $g$ vanish whenever one of
the last two entries belongs to $\Rad(TM)$. Since the screen metric is
non-degenerate, \eqref{eq:metric-K-screen} proves (1)--(3) are
equivalent. Under the same hypotheses, \eqref{eq:K-screen} becomes
$2K(X,W)=\nabla_X^{S}W-\nabla_X^{*S}W$, proving
(3)$\Leftrightarrow$(4).
\end{proof}

\begin{corollary}\label{cor:metric-rigidity}
Consider a lightlike hypersurface $M$ of
$(\overline{M},\overline{g},\overline{\nabla},\overline{\nabla}^{*})$
with induced connections $\nabla$ and $\nabla^{*}$. Suppose that both
$\Rad(TM)$ and $S(TM)$ are parallel for these connections and that
$\nabla$ is a metric connection. Then
\begin{equation}\label{eq:K-metric-radical}
K(X,Y)=\eta(X)\eta(Y)K(\xi,\xi)
\end{equation}
for all $X,Y\in\Gamma(TM)$. Consequently,
\begin{equation}\label{eq:connection-rigidity-equivalences}
\nabla=\nabla^{*}
\quad\Longleftrightarrow\quad
K(\xi,\xi)=0
\quad\Longleftrightarrow\quad
\tau(\xi)=\tau^{*}(\xi)
\quad\Longleftrightarrow\quad
\tau=\tau^{*}.
\end{equation}
\end{corollary}

\begin{proof}
Theorem~\ref{thm:metric-connections} gives $K(X,W)=0$ for every
$W\in\Gamma(S(TM))$. Hence \eqref{eq:K-block-decomposition} reduces to
\eqref{eq:K-metric-radical}. Under the same hypotheses,
\eqref{eq:K-xi} gives
$K(\xi,\xi)=-\frac12\{\tau(\xi)-\tau^{*}(\xi)\}\xi$.
The equivalences in \eqref{eq:connection-rigidity-equivalences} now follow
from \eqref{eq:connection-decomposition}, \eqref{eq:K-metric-radical}, and
\eqref{eq:tau-difference-screen}.
\end{proof}

\begin{corollary}\label{cor:screen-leaves}
Let $M$ be a lightlike hypersurface of the indefinite statistical manifold
$(\overline{M},\overline{g},\overline{\nabla},\overline{\nabla}^{*})$
with induced connections $\nabla$ and $\nabla^{*}$. If $\Rad(TM)$ and
$S(TM)$ are parallel with respect to both induced connections and $\nabla$
is a metric connection, then
$S(TM)$ is integrable, and on each integral manifold of $S(TM)$ the
restrictions of $\nabla^{S}$ and $\nabla^{*S}$ coincide with the
Levi--Civita connection of the metric induced by $g$.
\end{corollary}

\begin{proof}
Since $S(TM)$ is parallel with respect to $\nabla$ and $\nabla$ is
torsion-free, for $U,V\in\Gamma(S(TM))$ we have
$[U,V]\in\Gamma(S(TM))$. Thus $S(TM)$ is integrable.
Theorem~\ref{thm:metric-connections} gives
$\nabla^{S}=\nabla^{*S}$, and \eqref{eq:screen-duality} shows that this
common torsion-free connection is a metric connection on each integral
manifold. Hence it is the Levi--Civita connection of the induced
non-degenerate metric.
\end{proof}

Let $R$, $R^{*}$, and $R^{0}$ denote the curvature tensors of
$\nabla$, $\nabla^{*}$, and $\nabla^{0}$, respectively, with
$R(X,Y)Z=\nabla_X\nabla_YZ-\nabla_Y\nabla_XZ-\nabla_{[X,Y]}Z$ and
analogously for $R^{*}$ and $R^{0}$. Following the convention in
\cite{BahadirTripathi2023}, we write
\begin{equation}\label{eq:dtau-convention}
2d\tau(X,Y)=X(\tau(Y))-Y(\tau(X))-\tau([X,Y]),
\end{equation}
and analogously for $\tau^{*}$.

\begin{theorem}\label{thm:curvature-preservation}
Suppose that $M$ is a lightlike hypersurface of
$(\overline{M},\overline{g},\overline{\nabla},\overline{\nabla}^{*})$, and
let $\nabla$ and $\nabla^{*}$ denote the induced connections on $M$.
Assume that both $\Rad(TM)$ and $S(TM)$ are parallel with respect to these
connections. Then, for all $X,Y,Z\in\Gamma(TM)$,
\begin{align}
R(X,Y)PZ&=P(R(X,Y)Z),\label{eq:R-commutes-P}\\
R^{*}(X,Y)PZ&=P(R^{*}(X,Y)Z),\label{eq:Rstar-commutes-P}\\
R^{0}(X,Y)PZ&=P(R^{0}(X,Y)Z),\label{eq:R0-commutes-P}
\end{align}
and
\begin{align}
R(X,Y)\xi&=-2d\tau(X,Y)\xi,\label{eq:R-radical}\\
R^{*}(X,Y)\xi&=-2d\tau^{*}(X,Y)\xi,\label{eq:Rstar-radical}\\
R^{0}(X,Y)\xi&=-\{d\tau(X,Y)+d\tau^{*}(X,Y)\}\xi.
\label{eq:R0-radical}
\end{align}
Consequently, $R$, $R^{*}$, and $R^{0}$ preserve both $S(TM)$ and
$\Rad(TM)$.
\end{theorem}

\begin{proof}
By Theorem~\ref{thm:both-distributions}, $P$ is parallel with respect to
$\nabla$, $\nabla^{*}$, and $\nabla^{0}$. Their curvature operators
therefore commute with $P$, giving
\eqref{eq:R-commutes-P}--\eqref{eq:R0-commutes-P}. The same theorem and
\eqref{eq:xi-derivative}--\eqref{eq:xi-derivative-star} give
$\nabla_X\xi=-\tau(X)\xi$ and
$\nabla_X^{*}\xi=-\tau^{*}(X)\xi$. Averaging these identities gives
\[
\nabla_X^{0}\xi=-\frac12\{\tau(X)+\tau^{*}(X)\}\xi.
\]
The definitions of the three curvature tensors, together with
\eqref{eq:dtau-convention}, yield
\eqref{eq:R-radical}--\eqref{eq:R0-radical}. The last assertion follows
from $S(TM)=\operatorname{Im}P$ and $\Rad(TM)=\ker P$.
\end{proof}

\begin{corollary}\label{cor:radical-curvature}
Let $M$, with induced connections $\nabla$ and $\nabla^{*}$, be a lightlike
hypersurface of
$(\overline{M},\overline{g},\overline{\nabla},\overline{\nabla}^{*})$.
If $\Rad(TM)$ and $S(TM)$ are parallel with respect to both induced
connections, then $R(X,Y)$ vanishes on $\Rad(TM)$ for all
$X,Y$ if and only if $d\tau=0$. Likewise, $R^{*}(X,Y)$ vanishes on
$\Rad(TM)$ for all $X,Y$ if and only if $d\tau^{*}=0$, while
$R^{0}(X,Y)$ vanishes on $\Rad(TM)$ for all $X,Y$ if and only if
$d\tau+d\tau^{*}=0$.
\end{corollary}

\begin{corollary}\label{cor:metric-curvature-coincidence}
Consider a lightlike hypersurface $M$ of the indefinite statistical manifold
$(\overline{M},\overline{g},\overline{\nabla},\overline{\nabla}^{*})$,
with $\nabla$ and $\nabla^{*}$ induced by the ambient dual connections.
Suppose that $\Rad(TM)$ and $S(TM)$ are parallel for both induced
connections and that $\nabla$ is a metric connection. Then
\begin{equation}\label{eq:screen-curvature-coincidence}
R(X,Y)W=R^{*}(X,Y)W=R^{0}(X,Y)W
\end{equation}
for all $X,Y\in\Gamma(TM)$ and $W\in\Gamma(S(TM))$. Moreover,
\begin{equation}\label{eq:full-curvature-coincidence}
R=R^{*}
\quad\Longleftrightarrow\quad
d\tau=d\tau^{*},
\end{equation}
and whenever these equivalent conditions hold,
$R=R^{*}=R^{0}$.
\end{corollary}

\begin{proof}
Theorem~\ref{thm:metric-connections} gives
$\nabla^{S}=\nabla^{*S}$, and their common value is also the screen part of
$\nabla^{0}$. Since $S(TM)$ is parallel with respect to all three induced
connections, their curvature operators therefore agree on screen vector
fields, proving \eqref{eq:screen-curvature-coincidence}. By
\eqref{eq:R-radical} and \eqref{eq:Rstar-radical}, equality $R=R^{*}$ is
equivalent to $d\tau=d\tau^{*}$. In this case
\eqref{eq:R0-radical} gives the same action on $\Rad(TM)$, while
\eqref{eq:screen-curvature-coincidence} gives equality on $S(TM)$; hence
$R=R^{*}=R^{0}$.
\end{proof}

The following example shows that both distributions may be parallel for the
two induced connections even when neither connection is a metric connection.

\begin{example}\label{ex:nonmetric-parallel}
Let $\overline M=\mathbb R^{4}$ with coordinates
$(x_0,x_1,x_2,x_3)$ and semi-Euclidean metric
\[
\overline g=-dx_0^{2}+dx_1^{2}+dx_2^{2}+dx_3^{2}.
\]
Let $\overline\nabla^{0}$ be its Levi--Civita connection. Writing
$\partial_i=\partial/\partial x_i$, put
$E=\partial_{1}$, $F=\partial_{2}$, and $\theta=dx_1$, and define
\begin{equation}\label{eq:example-K}
\overline K(X,Y)=x_2\theta(X)\theta(Y)E,
\qquad
\overline\nabla=\overline\nabla^{0}+\overline K,
\qquad
\overline\nabla^{*}=\overline\nabla^{0}-\overline K.
\end{equation}
Since
$\overline g(\overline K(X,Y),Z)
=x_2\theta(X)\theta(Y)\theta(Z)$ is symmetric in $X,Y,Z$, the two
connections in \eqref{eq:example-K} are torsion-free and dual with
respect to $\overline g$. Hence
$(\overline M,\overline g,\overline\nabla,\overline\nabla^{*})$ is an
indefinite statistical manifold.

Consider $M=\{x_0=x_3\}$. It is a lightlike hypersurface with
\[
\xi=\partial_0+\partial_3,
\qquad
N=\frac12(-\partial_0+\partial_3),
\qquad
S(TM)=\operatorname{span}\{E,F\}.
\]
These fields satisfy \eqref{eq:null-frame}. For tangent $X,Y$,
$\overline K(X,Y)$ is screen-valued, whereas
$\overline K(X,\xi)=\overline K(X,N)=0$. Since the coordinate vector
fields are parallel with respect to the flat Levi--Civita connection
$\overline\nabla^{0}$, the Gauss--Weingarten formulas give
\[
A_{\xi}=A_{\xi}^{*}=0,
\qquad C=C^{*}=0,
\qquad \tau=\tau^{*}=0.
\]
Thus $\Rad(TM)$ and $S(TM)$ are parallel with respect to both induced
connections.

The induced difference tensor is not zero: $K(E,E)=x_2E$. Hence, on any
open subset where $x_2\neq0$,
\[
(\nabla_Eg)(E,E)=-2x_2,
\qquad
(\nabla_E^{*}g)(E,E)=2x_2,
\]
so neither induced connection is a metric connection. Also,
$\nabla_EE=x_2E$, $\nabla_FE=0$,
$\nabla_E^{*}E=-x_2E$, and $\nabla_F^{*}E=0$, which gives
\[
R(F,E)E=E,
\qquad
R^{*}(F,E)E=-E.
\]
Thus the curvature tensors of the induced connections need not vanish; they
preserve the screen distribution as asserted in Theorem~\ref{thm:curvature-preservation},
while
$R(X,Y)\xi=R^{*}(X,Y)\xi=R^{0}(X,Y)\xi=0$.
\end{example}

\section{Conclusion}\label{sec:conclusion}

We obtained necessary and sufficient conditions for the radical and screen
distributions of a lightlike hypersurface to be parallel with respect to both
induced connections. The conditions were expressed in terms of the screen
shape operators and the local screen fundamental forms, and were compared by
using the mean induced connection and the tangential difference tensor. In
particular, the simultaneous conditions for the two distributions are
equivalent to the screen projection being parallel with respect to both
induced connections, or, equivalently, to its being parallel with respect to
$\nabla^{0}$ together with the commutation relation
$K(X,PY)=PK(X,Y)$.

Under these conditions, the mixed radical--screen part of $K$ vanishes
and $K$ splits into a screen--screen term and a radical--radical term. The
connection $\nabla$ is a metric connection if and only if $\nabla^{*}$ is a
metric connection, and this occurs precisely when the screen--screen term
vanishes. In that case, the two induced connections can differ only in the
radical--radical direction; they coincide exactly when $\tau=\tau^{*}$. The
common screen connection is the Levi--Civita connection on each integral
manifold of $S(TM)$.

We also showed that the curvature tensors of $\nabla$, $\nabla^{*}$, and
$\nabla^{0}$ preserve the radical and screen distributions and determined
their action on the radical direction. In the metric case, the three curvature
tensors agree on screen directions, while $R=R^{*}$ is equivalent to
$d\tau=d\tau^{*}$; under this condition all three curvature tensors coincide.
The final example shows that the two distributions may be parallel for both
induced connections even though neither connection is a metric connection and
the induced curvature tensors are nonzero.



\begin{thebibliography}{99}

\bibitem{Amari1985}
S.-I. Amari,
\textit{Differential-geometrical methods in statistics},
Lecture Notes in Statistics, vol.~28, Springer-Verlag, New York, 1985.

\bibitem{Lauritzen1987}
S. L. Lauritzen,
Statistical manifolds,
in \textit{Differential Geometry in Statistical Inference},
IMS Lecture Notes--Monograph Series, vol.~10,
Institute of Mathematical Statistics, 1987, pp.~163--216.

\bibitem{Vos1989}
P. W. Vos,
Fundamental equations for statistical submanifolds with applications
to the Bartlett correction,
\textit{Ann. Inst. Statist. Math.} \textbf{41} (1989), no.~3, 429--450.

\bibitem{Furuhata2009}
H. Furuhata,
Hypersurfaces in statistical manifolds,
\textit{Differential Geom. Appl.} \textbf{27} (2009), no.~3, 420--429.

\bibitem{ONeill1983}
B. O'Neill,
\textit{Semi-Riemannian geometry with applications to relativity},
Pure and Applied Mathematics, vol.~103,
Academic Press, New York, 1983.

\bibitem{DuggalBejancu1996}
K. L. Duggal and A. Bejancu,
\textit{Lightlike submanifolds of semi-Riemannian manifolds and applications},
Mathematics and Its Applications, vol.~364,
Kluwer Academic Publishers, Dordrecht, 1996.

\bibitem{DuggalJin2007}
K. L. Duggal and D. H. Jin,
\textit{Null curves and hypersurfaces of semi-Riemannian manifolds},
World Scientific Publishing Co., Hackensack, NJ, 2007.

\bibitem{DuggalSahin2010}
K. L. Duggal and B. \c{S}ahin,
\textit{Differential geometry of lightlike submanifolds},
Frontiers in Mathematics, Birkh\"auser, Basel, 2010.

\bibitem{JainSinghKumar2020}
V. Jain, A. P. Singh and R. Kumar,
On the geometry of lightlike submanifolds of indefinite statistical
manifolds,
\textit{Int. J. Geom. Methods Mod. Phys.} \textbf{17} (2020), 2050099,
\url{https://doi.org/10.1142/S0219887820500991}.

\bibitem{KazemiSalahvarzi2020}
M. B. Kazemi Balgeshir and S. Salahvarzi,
Lightlike submanifolds of semi-Riemannian statistical manifolds,
\textit{Balkan J. Geom. Appl.} \textbf{25} (2020), no.~2, 52--65.

\bibitem{RaniKaur2021}
V. Rani and J. Kaur,
Lightlike hypersurfaces of an indefinite statistical manifold,
\textit{Adv. Appl. Math. Sci.} \textbf{20} (2021), no.~9, 1995--2013.

\bibitem{Bahadir2021}
O. Bahad{\i}r,
On lightlike geometry of indefinite Sasakian statistical manifolds,
\textit{AIMS Math.} \textbf{6} (2021), no.~11, 12845--12862,
\url{https://doi.org/10.3934/math.2021741}.

\bibitem{BahadirEtAl2022}
O. Bahad{\i}r, A. N. Siddiqui, M. G\"ulbahar and A. H. Alkhaldi,
Main curvatures identities on lightlike hypersurfaces of statistical
manifolds and their characterizations,
\textit{Mathematics} \textbf{10} (2022), no.~13, 2290,
\url{https://doi.org/10.3390/math10132290}.

\bibitem{BahadirTripathi2023}
O. Bahad{\i}r and M. M. Tripathi,
Geometry of lightlike hypersurfaces of a statistical manifold,
\textit{WSEAS Trans. Math.} \textbf{22} (2023), 466--474,
\url{https://doi.org/10.37394/23206.2023.22.52}.

\bibitem{BhattiKaur2024}
S. Bhatti and J. Kaur,
Lightlike hypersurfaces of an indefinite $(\alpha,\beta)$-type almost
contact metric statistical manifold with an $(l,m)$-type connection,
\textit{J. Indones. Math. Soc.} \textbf{30} (2024), no.~3, 398--420.

\bibitem{AhmadAlamAli2025}
M. Ahmad, M. Alam and S. Ali,
Screen semi-invariant lightlike submanifolds of an indefinite Kenmotsu
statistical manifold,
\textit{Commun. Korean Math. Soc.} \textbf{40} (2025), no.~4, 901--916,
\url{https://doi.org/10.4134/CKMS.c250017}.

\end{thebibliography}
\end{document}